%% file: one-sided-equidistribution.tex
\documentclass[reqno]{amsart}
\usepackage{amsmath}
\usepackage{amsfonts}
\usepackage{amssymb}
\usepackage{amstext}
\usepackage{amsbsy}
\usepackage{amsopn}
\usepackage{amsthm}
\usepackage{amsxtra}
\usepackage{upref}
\usepackage{fancyhdr,amsmath, graphicx, psfrag, color ,amsfonts,layout, amssymb}
\usepackage{euscript,epsfig}

\newtheorem{thm}{Theorem}[section]

\newtheorem{lem}[thm]{Lemma}
\newtheorem{prop}[thm]{Proposition}
\newtheorem{fact}[thm]{Fact}
\newtheorem{question}[thm]{Question}
\newtheorem{rem}[thm]{Remark}
\theoremstyle{definition}
\newtheorem{defn}[thm]{Definition}
\theoremstyle{remark}

\numberwithin{equation}{section}

\newtheorem*{example*}{Example}

\newcommand{\eop}{\hfill $\Box$}

\def\text#1{\textrm{#1}}

\def\d{\delta}

\def\D{\mathbb D}

\def \R{\mathbb R}
\def \N{{\mathbb N}}

\def \Z{\mathbb Z}

\def \H{\mathbb H}

\def\D{\mathbb D}

\def\({\biggl(}
\def\){\biggr)}
\def\<{\bold\langle}
\def\>{\bold\rangle}

\title[One-Sided horocycles]{
Density and equidistribution of one-sided 
horocycles of a geometrically finite hyperbolic surface}

\author{Barbara Schapira}
\address{B. Schapira,
L.A.M.F.A.,  C.N.R.S. UMR 6140,
Universit\'e Picardie Jules Verne,
33 rue St Leu
80000 Amiens, France}
\email{Barbara.Schapira@u-picardie.fr}
\begin{document}

\begin{abstract}
On geometrically finite negatively curved surfaces, we give necessary and
sufficient conditions for a one-sided horocycle $(h^su)_{s\ge 0}$ to be
dense in the nonwandering set of the geodesic flow. 
We prove that all dense one-sided orbits $(h^su)_{s\ge 0}$ are
equidistributed, extending results of 
 \cite{Burger} and \cite{Scha2} where 
symmetric horocycles $(h^su)_{-R\le s\le R}$ were considered. 
\end{abstract}
\maketitle


\section{Introduction}

Hedlund \cite{Hedlund} proved that the horocyclic flow $(h^s)_{s\in\R}$ 
on the unit tangent bundle 
of a finite volume hyperbolic surface is minimal, that is all nonperiodic 
orbits $(h^sv)_{s\ge 0}$ (called in \cite{Hedlund} "right-semihorocycles", 
and here positive half-horocycles) 
are dense. 

On {\em geometrically finite surfaces}, i.e.
 surfaces whose fundamental group is finitely generated, 
it is known (see \cite{Eberlein},  \cite{Dalbo})
 that all nonwandering 
and non periodic orbits of the horocyclic flow are {\em dense} in the sense that the closure 
of $(h^sv)_{s\in\R}$ contains the nonwandering set of the geodesic flow. 
On general hyperbolic surfaces, keeping  this definition of ``dense'' 
horocyclic orbit, we know (\cite {Eberlein} \cite{Dalbo} \cite{Starkov})
 how to characterize dense orbits\,: a horocycle is dense iff 
it is centered at a horospherical point. 
 
However, 
as soon as the fundamental group of the surface is of the second kind, 
i.e. its limit set is strictly included in the boundary at infinity $S^1$ (see section 2), 
we can easily find horocycles $(h^s u)_{s\in\R}$ 
that are globally dense in the nonwandering set $\Omega$ 
of the geodesic flow (in the sense that $\overline{(h^su)_{s\in\R}}\supset \Omega$), 
but with one side 
dense and the other not.

In this note, answering a question of  O. Sarig,
we characterize these horocycles with one side dense and the other not. 
If $u\in T^1 S$, and $\tilde{u}$ is any of its lifts on the unit tangent bundle $T^1\D$ 
of the hyperbolic disc, 
we denote by $u^-\in S^1$ (resp. $u^+$) the negative (resp. positive) 
endpoint in the boundary $S^1=\partial \D$ of the geodesic 
line defined by $\tilde{u}$. 
We prove:

\begin{thm}\label{density} Let $S$ be a geometrically finite hyperbolic
  surface. 
Let $u\in T^1 S$ be s.t. its full unstable horocyclic orbit $(h^su)_{s\in\R}$ is dense in
  $\Omega$. Then the positive half-horocycle $(h^s u)_{s\ge 0}$ 
is dense in $\Omega$ iff $u^-$ is not the
  first endpoint of an interval of $S^1\setminus \Lambda_{\Gamma}$ 
  (where the circle $S^1$ is oriented in the counterclockwise direction). 
\end{thm}

On general hyperbolic surfaces, this theorem remains valid 
for vectors $u\in T^1 S$ that are periodic for the geodesic flow 
(see proposition \ref{periodic}). 

In fact, inspsired by ideas of \cite{C}, we introduce the notion of {\em
  right horospherical vector}, and prove (proposition \ref{right_iff_dense}) 
that on general hyperbolic surfaces, a positive half-horocycle 
 $(h^su)_{s\ge 0}$ is dense iff
  $u$ is a right horospherical vector. 
  We deduce Theorem \ref{density} from the fact that on geometrically finite surfaces,
 right horospherical vectors are easy to characterize. \\

Our initial motivation was the study of equidistribution properties of horocycles. 
Furstenberg's unique ergodicity result \cite{F} for  the horocyclic flow ensures 
that on the unit tangent bundle $T^1 S$ of a compact hyperbolic surface, 
all horocyclic orbits are equidistributed towards the unique $(h^s)$-invariant
 measure $\lambda$: for all $u\in T^1 S$ and $f:T^1 S\to \R$ continuous,  
$\displaystyle \frac{1}{T}\int_0^Tf\circ h^s u\,ds\to \int_{T^1
  S}f\,d\lambda$, 
where $\lambda$ is the Liouville measure. 
Of course, the same result holds for $(h^{s}u)_{s\le 0}$. 

This result was extended by Dani and Smillie \cite{DS} to finite volume
hyperbolic surfaces: 
all nonperiodic one-sided orbits $(h^su)_{s\ge 0}$ are equidistributed towards $\lambda$.

On geometrically finite hyperbolic surfaces, there is \cite{Roblin} \cite{Burger} 
a unique $(h^s)$-invariant ergodic measure $m$ 
that has full support in the nonwandering set of $(h^s)$; and it is infinite. 
Therefore, as in Hopf ergodic theorem, one  considers ratios 
$\displaystyle\frac{\int_{-T}^T f\circ h^s u\,ds}{\int_{-T}^T g\circ h^s u\,ds}$ 
and one can prove \cite{Burger}\cite{Scha2} that 
they converge to  $\displaystyle\frac{\int_{T^1 S}f\,dm}{\int_{T^1 S}g\,dm}$ 
for all continuous functions $f, g:T^1 S\to \R$ with compact support, 
and all nonwandering and non periodic vectors $u\in T^1 S$. 
In these two articles, equidistribution is obtained for symmetric 
horocycles $(h^s u)_{-T\le s\le T}$ only, 
and not for one-sided horocycles $(h^s u)_{0\le s\le T}$.  
Symmetric averages are very natural from a geometric point of view, but not from the ergodic point of view, 
where a difference of behaviour between the negative and the positive orbit is an interesting phenomenon. 

In theorem \ref{density}, we characterized dense horocycles that 
have one side dense and the other not. 
 For these horocycles, one cannot hope equidistribution of both one-sided orbits. 
However, according to Hopf ergodic theorem, almost all one-sided horocycles
 should be equidistributed towards $m$. 

On geometrically finite hyperbolic surfaces, 
 the above phenomenon of dense horocycles with a nondense side 
is the only obstruction to the equidistribution of one-sided horocycles. 
Indeed, with  methods of \cite{Scha1} and \cite{Scha2}, we get:  

\begin{thm}\label{equidistribution} Let $S$ be a geometrically finite surface, 
and $u\in T^1 S$ 
such that $(h^su)_{s\ge 0}$ is dense in the nonwandering set of the geodesic flow. 
Then $(h^su)_{s\ge 0}$ is equidistributed towards the unique invariant measure $m$ 
which has full support in the nonwandering set of $(h^s)_{s\in\R}$. 

In other words, for all continuous functions with compact support $f,g:T^1 S\to \R$, 
with $\int_{T^1 S} g\,dm>0$, we have
$$
\frac{\int_0^T f\circ h^s u\,ds }{\int_0^T g\circ h^s u\,ds} \rightarrow
\frac{\int_{T^1 S} f\,dm}{\int_{T^1 S} g\,dm}\,,\quad\mbox{when }\quad T\to
+\infty\,.
$$
\end{thm}
Note that periodic orbits are obviously equidistributed to the Lebesgue measure on the orbit. 
Of course, theorem \ref{equidistribution} also holds for negative orbits $(h^su)_{s\le 0}$.

Most results extend to  surfaces of variable negative curvature.
 However, to avoid too many preliminaries, we postpone 
the discussion on such surfaces to the end of the paper. 

Section  \ref{Preliminaries} is devoted to preliminaries. 
Theorem \ref{density} is proved in section \ref{Proof-of-density}, where we also discuss the case of geometrically infinite surfaces,  
and theorem \ref{equidistribution} in section \ref{Proof-of-equidistribution}. 

I thank warmly O. Sarig for the question at the origine of this note,  
F. Dal'bo for her comments on the first version of this work, 
Yves Coudene for several helpful discussions, and all
members of the former french ANR project on infinite ergodic theory for our fruitful work together.


\section{Preliminaries}\label{Preliminaries}


\subsection*{Hyperbolic geometry}
The hyperbolic  disc $\D=D(0,1)$ is endowed with the metric 
$\frac{1}{4}\frac{dx^2}{(1-|x|^2)^2}$.
Let $o$ be the origin of the disc. Denote by $\pi:T^1\D\to \D$ the canonical projection. 
The {\em boundary at infinity}  is
$S^1=\partial\D$.

The map $z\in\D\mapsto  \frac{i(1+z)}{1-z}$ is an isometry between $\D$ 
with the above metric and the upper half plane $\H=\R\times \R_+^*$ 
endowed with the hyperbolic metric 
$\frac{dx^2+dy^2}{y}$. 
Therefore, the group  of isometries preserving orientation of $\D$ 
identifies with $PSL(2,\R)$ acting
by homographies on $\H=\R\times \R_+^*$. 
It acts simply transitively on the unit tangent bundle $T^1\D$, 
so that we identify these two spaces through the  map which sends the unit
vector $(1,0)$ tangent to $\D$ at $o=(0,0)$ on the identity element of $PSL(2,\R)$.

The  {\em Busemann cocycle } is the continuous map defined on $S^1\times\D^2$ by
$$
\beta_{\xi}(x,y):=\lim_{z\to\xi}\left(d(x,z)-d(y,z)\right)\,.
$$
Define the map
$\displaystyle
v\in T^1\D\mapsto (v^-,v^+,\beta_{v^-}(\pi(v),o))\,,$
where $v^{\pm}$  are the endpoints in 
$S^1$ of the geodesic defined by $v$, and  $\pi(v)\in\D$ is the basepoint in $S$ of  $v$. 
It defines a homeomorphism between $T^1\D$ and 
 $\partial^2\D\times\R:=S^1\times S^1\setminus\mbox{Diagonal}\times\R$.

Let $\Gamma$ be a discrete subgroup of  $PSL(2,\R)$. 
Its {\em limit set} is  
$\Lambda_{\Gamma}=\overline{\Gamma o}\setminus \Gamma o \subset S^1$. The
group $\Gamma$ acts properly discontinuously on the {\em
  ordinary set} $S^1\setminus\Lambda_{\Gamma}$, which is a countable union of
intervals.

A point $\xi\in\Lambda_{\Gamma}$ is a  {\em radial} limit point 
 if it is the  limit of a sequence
$(\gamma_n .o)$ of points of $\Gamma . o$ that stay at bounded distance 
of the geodesic ray $[o\xi)$ joining $o$ 
to $\xi$.
Let $\Lambda_{\rm rad}$ denote the {\em radial limit set}.

The point $\xi\in\Lambda_{\Gamma}$ is {\em horospherical} if any horoball
centered at $\xi$ contains infinitely many points of $\Gamma.o$. In
particular, $\Lambda_{\rm rad}$ is included in the horospherical set
$\Lambda_{\rm hor}$.

An isometry of $PSL(2,\R)$ is  {\em hyperbolic} if it fixes exactly two points of  
 $S^1$, it is
 {\em parabolic} if it fixes exactly one point of $S^1$, and {\em elliptic} in the other cases.
Let  $\Lambda_{\rm p}\subset \Lambda_{\Gamma}$ denote the set of {\em parabolic} 
limit points, that is the points of $\Lambda_{\Gamma}$ fixed by a parabolic isometry of
 $\Gamma$.

Any hyperbolic surface is the quotient  $S=\Gamma\backslash \D$ of $\D$ by a
 discrete subgroup $\Gamma$ of $PSL(2,\R)$ without elliptic element. 
Its unit tangent bundle  $T^1S$ identifies with $\Gamma\backslash PSL(2,\R)$.

In this note, we always assume $\Gamma$ be {\em nonelementary},
 that is $\#\Lambda_\Gamma=+\infty$. 
Moreover, we are mainly interested in {\em geometrically finite surfaces} $S$,
 i.e. surfaces whose fundamental group  $\Gamma$ is finitely generated. 
In such cases, the limit set $\Lambda_{\Gamma}$ is the disjoint union of $\Lambda_{\rm rad}$ 
and $\Lambda_{\rm p}$  \cite{Bowditch}.
Moreover, the surface is a disjoint union of a compact part $C_0$, 
finitely many cusps (isometric to $\{z\in\H,\,\, Im\, z\ge cst\}/\{z\mapsto z+1\}$), 
and finitely many 'funnels' 
(isometric to $\{z\in\H,\,\, Re(z)\ge 0,\, 1\le |z|\le a\}/\{z\mapsto az\}=
\{z\in\H,\,\, Re(z)\ge 0\}/\{z\mapsto az\}$, for some $a>1$. 

When $S$ is compact, $\Lambda_{\Gamma}=\Lambda_{\rm rad}=S^1$. 
It is said {\em convex-cocompact}  when it is a geometrically finite surface without cusps. 
In this case, $\Lambda_{\Gamma}=\Lambda_{rad}$ is strictly included in $S^1$
 and $\Gamma$ acts cocompactly 
on  the set 
$(\Lambda_{\Gamma}\times \Lambda_\Gamma)\setminus \mbox{Diagonal}\times\R\subset T^1\D$. 
(We identify now the two homeomorphic spaces $T^1\D$ and 
$S^1\times S^1\setminus\mbox{Diagonal}\times \R$.)
When  $S$  has finite volume, there are no funnels and $\Lambda_{\Gamma}=S^1$. 

\subsection*{Geodesic and horocycle flows}

A hyperbolic geodesic in $\D$ is a diameter or a half-circle orthogonal to $S^1$. 
A {\em horocycle} of $\D$ is a circle tangent to $S^1$.
 It can also be defined as a level set of a Busemann function.  
A {\em horoball} is the (euclidean) disc delimited by a horocycle. 
A vector $v\in T^1\D$ is tangent to a unique geodesic of $\D$. 
Moreover, it is orthogonal to exactly two horocycles passing through its basepoint $\pi(v)$, 
and tangent to $S^1$ respectively at $v^+$ and $v^-$. 
The set of vectors $w\in T^1 \D$ such that $w^-=v^-$ and 
based on the same horocycle tangent to $S^1$ at $v^-$ 
is the {\em strong unstable horocycle} or strong unstable manifold 
$W^{su}(v)\subset T^1 \D$ of $v$
The  {\em strong stable manifold} $W^{ss}(v)$ is defined in the same way. 

The {\em geodesic flow}  $(g^t)_{t\in\R}$  acts on $T^1\D$ by moving
 a vector $v$ of a distance $t$ along its geodesic. 
In the identification of
 $T^1\D$ with $PSL(2,\R)$, this flow corresponds to the right action by
 the one-parameter subgroup 
 $$
 \left\{a_t:=\left(\begin{array}{cc}
e^{t/2} & 0 \\
0 & e^{-t/2} \\
\end{array} \right),\,t\in\R
\right\}.
$$

The {\em strong unstable horocyclic flow}  $(h^s)_{s\in\R}$ acts on $T^1\D$ by moving a vector
  $v$ of a distance $|s|$ 
along its strong unstable horocycle. There are two possible orientations for this
  flow, 
and we consider the choice corresponding to the right action 
by the one parameter subgroup
 $$ 
\left\{n_s:=\left(\begin{array}{cc}
 1& 0 \\
s & 1 \\
\end{array} \right),\,s\in\R
\right\}$$
on $PSL(2,\R)$. This flow turns vectors along their strong unstable horocycle,
 so that $\{h^s v, \,s\in\R\}=W^{su}(v)$. 
The horocyclic orbits are the {\em strong unstable manifolds} of the geodesic 
flow in the sense that 
$$
W^{su}(v)=\{w\in T^1\D,\, d(g^{-t}v, g^{-t}w)\to 0 \mbox{ quand } t\to +\infty\}\,.
$$
Moreover, it satisfies
$$ g^t\circ h^s=h^{se^t}\circ g^t\,.$$

These two right-actions are well defined on the quotient space 
$T^1 S\simeq \Gamma\backslash PSL(2,\R)$. 
The nonwandering set $\Omega$ of the geodesic flow is the set 
$\Gamma\backslash (\Lambda_{\Gamma}^2\times  \R)$.
The horocyclic flow is topologically transitive (see \cite{Dalbo}) in the sense that
 there exists $u\in T^1 S$ such that 
$\overline{(h^s u)_{s\in\R}}\supset \Omega$. It allows to see that
the nonwandering set ${\mathcal E}$ of the horocyclic flow is the set 
$\Gamma\backslash (\Lambda_{\Gamma}\times S^1\times \R)$ of vectors 
such that $v^-\in \Lambda_{\Gamma}$. 

In our situation (nonelementary hyperbolic surfaces)
 we know that the length spectrum of the fundamental group
 $\Gamma$ of $S$ is nonarithmetic, that is the set $\{l(\gamma)\}$
 of lengths of closed geodesics generates a dense subgroup of $\R$.
 We will use this crucial fact in the sequel.

\subsection*{Local product structure of the geodesic flow}

The geodesic flow on the unit tangent bundle of any 
hyperbolic surface (including $\D$)
 is a {\em hyperbolic flow}. In particular,
 it has a (uniform) {\em local product structure}\,:
{\em for all $\varepsilon>0$, there exists $\delta>0$ s.t. if $d(u,v)\le \delta$, 
there is a vector $w=[u,v]$ in $W^{su}_\varepsilon(g^t u)\cap W^{ss}_\varepsilon(v)$, 
where $W^{ss}_\varepsilon(v)$ is the intersection of the strong stable horocycle of $v$ 
with the ball centered at $v$ of radius $\varepsilon$ and $|t|\le \varepsilon$. }

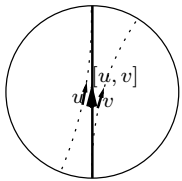
\begin{figure}[ht!]
\begin{center}\label{produitlocal}
\input{produitlocal.pstex_t}
\caption{Local product in the hyperbolic disc $\D$}
\end{center}
\end{figure}


\section{Density of positive half-horocycles}\label{Proof-of-density}

Recall that $W^{su}(v)=\{h^s v,\,s\in\R\}$ is compact iff $v^-\in \Lambda_p$,
and dense in $\Omega$ iff $v^-\in \Lambda_{hor}$ (see \cite{Dalbo}). 
Denote by $W^{su}_+(v)=\{h^s v,\,s\ge 0\}$ the positive half-horocycle. 

We suppose in the sequel that $S^1$ is oriented in the counterclockwise
direction. 

\subsection*{Geometry of funnels}

\begin{rem}\rm If the surface $S=\D/\Gamma$ has a funnel isometric 
to $\{z\in\H, \,Re(z)\ge 0\}/\{z\mapsto az\}$, with $a>1$,
 the geodesic line $Re(z)=0$ 
of $\D$ induces on the quotient the closed geodesic closing the funnel. 
Any geodesic line crossing this closed geodesic and entering into the funnel 
never returns back to the other side. 
In particular, the limit set $\Lambda_\Gamma$ does not 
intersect the right half line $\R_+^*$. 
\end{rem}

From this elementary remark, we deduce the following key facts.

\begin{fact} \rm On a geometrically finite hyperbolic manifold, the only
  points on the boundary of an interval of $S^1\setminus \Lambda_\Gamma$ are
  hyperbolic. More precisely, both extremities of such an interval are the
  endpoints $p^\pm$ of the axis of a lift of the closed geodesic bording the
  corresponding funnel. 
\end{fact}

\begin{fact}\rm Assume $S$ be geometrically finite. 
 If $v^- \in \Lambda_{hor}$ is the first endpoint of an interval of
  $S^1\setminus \Lambda_{\Gamma}$, then $W^{su}_+(v)$ is not dense in $\Omega$
  and $(g^{-t}v)_{t\ge 0}$ is asymptotic to the closed geodesic turning around
  a funnel. 
\end{fact}

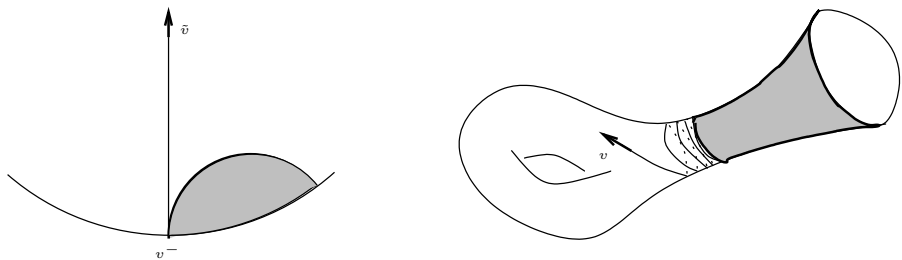
\begin{figure}[ht!]
\begin{center}\label{horocyclenondense}
\input{horocyclenondense.pstex_t}
\caption{A vector whose right horocycle is not dense in $\Omega$}
\end{center}
\end{figure}


\subsection*{Right horocyclic vectors and right horocyclic points}

If $v\in T^1\D$, we denote by $Hor(v)\subset \D$ the horoball centered at $v^-$ 
and containing  the base point of $v$ in its boundary. 
We denote by $Hor^+(v)\subset Hor(v)$ the ``right part'' of the horoball, 
i.e. the set of basepoints of vectors of $\cup_{t\ge 0}W^{su}_+(g^{-t}v)$. 

Fix a point $o\in \D$. 
If $S$ is a geometrically finite surface, we assume that $o$ 
belongs to a lift of the compact part of $S$.

In \cite{C}, a vector $v\in T^1S$ is called {\em horospherical} if there exists
$z\in\Omega$, $t_i\to +\infty$ and 
$v_i\in W^{su}(v)\cap\Omega$ s.t. $g^{-t_i}v_i\to z\in\Omega$. 
It
is equivalent to saying that $v^-\in\Lambda_{hor}$, 
that is that all horoballs centered at $v^-$ contain infinitely many points 
of the orbit $\Gamma.o$ (see  lemma \ref{right-horocyclic} below for a proof).

\begin{defn} If $v\in T^1\D$, and $\alpha>0$, we define the {\em cone} 
of width $\alpha$ around $v$ as the set 
$\mathcal{C}(v,\alpha)$ of points at distance at most $\alpha$ from the 
geodesic ray $(g^{-t}v)_{t\ge 0}$ inside the horoball $Hor(v)$.  
\end{defn}

\begin{defn} Let $S$ be a nonelementary hyperbolic surface. 
A vector $v\in T^1 S$ is a {\em right horocyclic vector} 
if for a lift $\tilde{v}\in T^1 S$, for all $\alpha>0$ and $D>0$, 
the orbit $\Gamma.o$ intersects 
the right horoball $Hor^+(g^{-D}\tilde{v})$ minus the cone $C(g^{-D}\tilde{v}, \alpha)$. 
\end{defn}

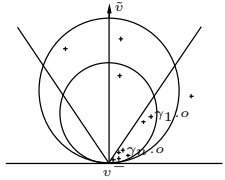
\begin{figure}[ht!]
\begin{center}\label{horocyclenondenseright-horocyclic}
\input{right-horocyclic.pstex_t}
\caption{Lift of a right-horocyclic vector}
\end{center}
\end{figure}

Of course if $v$ is a right horocyclic vector, then $v^-$ is horospherical,
 and equivalently $v$ is a horospherical vector in the sense of \cite{C}.

\begin{rem}\rm This definition depends only of $v^-$ (indeed, if $w$ is another vector 
with $v^-=w^-$, any cone around w is included in a cone around $v$). 
A point $\xi\in \Lambda_\Gamma$ which is the negative endpoint of a right 
horocyclic vector will therefore be called a {\em right horocyclic point}. 
\end{rem}

\begin{lem}\label{right-horocyclic} Let $S$ be a non elementary hyperbolic surface. 
A vector   $v\in T^1 S$ is a right horocyclic vector 
if and only if there exists $z\in\Omega$ such that 
 for all $\alpha$, there exists a sequence $t_n\to +\infty$,  $v_n\in W^{su}_+(v)$ 
 s.t.  $g^{-t_n}v_n$ converges to $z\in \Omega$, 
but  $g^{-t_n}\tilde{v}_n\notin \mathcal{C}(\tilde{v},\alpha)$, 
where $\tilde{v}$ and $\tilde{v_n}$ are lifts resp. 
of $v$ and $v_n$ on the same  horocycle of $T^1 \D$.
 \end{lem}

The definition of right horocyclic vector is easier, but the above equivalent property 
will be more useful in the sequel. 

\begin{proof} Let us begin with the following elementary fact. 
\begin{fact}\label{uniformity} There exists $R>0$, such that for all $\xi\in\Lambda_\Gamma$, 
there exists $\eta\in\Lambda_\Gamma$, such that the geodesic $(\xi\eta)$ 
intersects the ball $B(o,R)$. 
\end{fact}
Indeed, assuming it is false, we could find a sequence $R_n\to \infty$, 
$\xi_n\in\Lambda_\Gamma$, $\xi_n\to\xi \in\Lambda_\Gamma$, 
s.t. for all $\eta\in\Lambda_\Gamma$, the distance $d(o,(\xi_n\eta))$ 
is greater than $R_n$. Passing to the limit, for $\eta\neq\xi$, 
we obtain $d(o,(\xi\eta))=+\infty$, which gives a contradiction. \\

Now, let $v$ be a right horocyclic vector. 
Let $D_n\to +\infty$, $\alpha_n\to +\infty$, and $\tilde{v}$ be a lift of $v$ to $T^1\D$. 
There exists a point $\gamma_n.o$ in 
$Hor^+(g^{-D_n}\tilde{v})\setminus \mathcal{C}(\tilde{v},\alpha)$. 
Using fact \ref{uniformity}, we can find $\eta\in\Lambda_\Gamma$, $\eta\neq v^-$, 
s.t. the geodesic $(\tilde{v}^-\eta)$ intersects the ball $B(\gamma_n.o,R)$. 
Choose a vector $\tilde{w}_n\in \widetilde{\Omega}\cap T^1 B(\gamma_n.o,R)$ 
tangent to this geodesic. 
It satisfies $w_n^-=v^-$, $w_n^+=\eta$, $\tilde{w}_n=g^{-t_n}\tilde{v}_n$, 
$t_n\ge D_n-R$, $\tilde{v}_n\in W^{su}_+(\tilde{v})$, 
and $\tilde{w}_n$ does not belong to the cone $\mathcal{C}(\tilde{v},\alpha_n-R)$. 
Passing to $T^1 S$, we get a sequence of vectors $w_n$ of the compact set 
$T^1 B(o,R)\cap \Omega$. Up to a subsequence, it converges to some $z\in\Omega$. 
We proved that there exists $z\in\Omega$, s.t. for all $\alpha>0$, 
there exists $t_n\to +\infty$,
and $v_n\in W^{su}_+(v)$ s.t. $g^{-t_n}v_n\to z$, 
and $g^{-t_n}\tilde{v}_n\notin \mathcal{C}(\tilde{v},\alpha)$. \\

Conversely, assume the existence of such a $z\in\Omega$. Fix $\alpha>0$ and $D>0$.  
Let $\rho=d(o,\pi(z))$, $\alpha>0$, and $\beta=\alpha+\rho+1$.
 There exists $t_n\to\infty$, $v_n\in W^{su}_+(v)$, $g^{-t_n}v_n\to z$,
 and $\tilde{v}_n\notin\mathcal{C}(\tilde{v},\beta)$. 
For $n$ large enough, $t_n\ge D+\rho+1$, and $d(g^{-t_n}v_n,z)\le 1$. 
There exists an element $\gamma_n.o\in B(\pi(g^{-t_n}\tilde{v}_n), \rho)$. 
By construction, this element is in 
$Hor^+(g^{-D}\tilde{v})\setminus \mathcal{C}(g^{-D}\tilde{v},\alpha)$. 
Thus, $\tilde{v}$ is a right horocyclic vector. 
\end{proof}


\subsection*{Proof of theorem \ref{density}}

We will prove 

\begin{prop}\label{right_iff_dense}  Let $S$ be a nonelementary hyperbolic 
surface. 
A vector  $\tilde{v}\in T^1\D$ is a right horocyclic vector if and only if
  $W^{su}_+(v)$ is dense in $\Omega$, where $v\in T^1S$ is the projection of $\tilde{v}$. 
\end{prop}

and 

\begin{lem}\label{right-horospherical-gf} Let $S$ be a nonelementary geometrically finite surface. 
If $v^-\in\Lambda_{\rm hor}$, $v$ is a right
  horocyclic point iff $v^-$ is not the first endpoint of an interval of
  $S^1\setminus \Lambda_{\Gamma}$. 
\end{lem}

Theorem \ref{density} is  an immediate consequence of these two results.
Let us now prove them. 

\begin{fact}\label{open} If $y\in \overline{W^{su}_+(x)}$, then $W^{su}_+(y)\subset
  \overline{W^{su}_+(x)}$. 
\end{fact}
\begin{proof} Evident with the parametrization of $W^{su}$ by the horocyclic flow. 
\end{proof}

For a vector $v\in T^1S$, we denote by $\tilde{v}$ a lift to $T^1\D$, 
and by $v^\pm\in S^1$ the enpoints of this lift on the boundary. 

\begin{prop}\label{periodic} Let $S$ be a nonelementary surface. 
If $p\in\Omega$ is a periodic vector for the geodesic flow,
 then $W^{su}_+(p)$ is dense in $\Omega$ if and only if
  $p^-$ is not the first endpoint of an interval of $S^1\setminus
  \Lambda_{\Gamma}$. 
\end{prop}

This result is valid on any nonelementary negatively curved surface,
 without geometrical
finiteness assumption. 

Recall first that on a nonelementary negatively curved surface,
 the {\em length spectrum} is
non arithmetic (see \cite{Dalbo}), that is the set of lengths of periodic
orbits $\{l(\gamma), <\gamma> \mbox{ periodic} \}$ generates a dense subgroup
of $\R$. 

\begin{proof} Note first that if $p\in T^1 S$ 
is a periodic vector for the geodesic flow and $p^-$ is 
the first endpoint of an interval $]p^-\eta[$ of 
$S^1\setminus \Lambda_\Gamma$, then $W^{su}_+(p)$ cannot be dense in $\Omega$. 
Indeed, the vectors of $W^{su}_+(p)$ pointing in $]p^-\eta[$ 
do not even belong to $\Omega$. 

Assume now that $p^-$ is not the first endpoint of an interval 
of $S^1\setminus \Lambda_\Gamma$. 
We follow \cite{C} almost verbatim. 

First, $g^{\R}W^{su}_+(p)$ is dense in $\Omega$ 
(see also \cite[Lemma1]{C}). 
Indeed, there exists $x\in \Omega$, s.t. $(g^{t}x)_{t\ge
  0}$ is dense in $\Omega$. 
Let $\tilde{x}$ (resp. $\tilde{p}$) be a lift of
$x$ (resp. $p$) to $T^1\D$, and
$x^+ $  its positive endpoint in $S^1$. 
The orbit $\Gamma.x^+$ is dense
in~$\Lambda_{\Gamma}$. 
As $p^-$ is not the first endpoint of an interval of
$S^1\setminus\Lambda_\Gamma$, we can find a sequence $x_n^+\in\Gamma.x^+$ 
converging to $p^-$, with
$x_n^+\ge p^-$ (in the counterclockwise order).

\begin{figure}[ht!]
\begin{center}\label{faiblestable}
\input{faiblestable.pstex_t}
\caption{Construction of a dense geodesic in the weak unstable manifold $W^{wu}(p)$}
\end{center}
\end{figure}
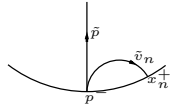

 If $n$ is large enough, the
unique vector $\tilde{v}_n$ of $W^{ss}(\tilde{p})\cap(p^-x_n^+)$ belongs to
the positive half-horocycle $W^{su}_+(\tilde{p})$, so that on $T^1 S$, 
$v_n\in W^{su}_+(p)$ and $g^{\R}v_n$ is dense in $\Omega$. Therefore, $g^\R W^{su}_+(p)$ is dense in $\Omega$. \\

Now, let us show that $W^{su}_+(p)$ is dense in $g^{\R}W^{su}_+(p)$. 
(We still follow \cite{C}). 
Fix $\varepsilon>0$, and a periodic vector $p_0\in T^1 S$, s.t.
 $\exists m,n\in\Z$, with $|ml(p)+nl(p_0)|<\varepsilon$.
Without loss of generality, assume $n\ge 0$. 
Let $\delta=\delta(\varepsilon,p_0)>0$ be the constant appearing in 
the local product structure property around $p_0$ (see end of section 2). 

As $p^-$  is not the first endpoint of an
interval of $S^1\setminus \Lambda_{\Gamma}$, we can
lift $p_0$ to $\tilde{p_0}$ in such a way that $p_0^+\in[p^-p^+]$.

Let $v\in W^{su}_+(p)\cap W^{ws}(p_0)$ be the vector obtained as the projection on $T^1 S$ of 
the unique $\tilde{v}$ of $W^{su}_+(\tilde{p})\cap W^{ws}(\tilde{p}_0)$. 
If $p_0^+$ is well chosen (i.e. close enough to $p^-$), 
we have $W^{su}_{2\varepsilon}(v)\subset W^{su}_+(p)$. 
The geodesic orbit $(g^tv)_{t\in\R}$ is negatively asymptotic to the periodic orbit of $p$,
 positively asymptotic to the periodic orbit of~$p_0$. 

The end of the proof consists in using once again the local product structure to construct 
an orbit which is negatively asymptotic to the orbit of $p$, 
positively asymptotic to the orbit of $p$, and ``in the middle''  goes from the orbit 
of $p$ to the orbit of $p_0$ following the orbit of $v$, turns a certain number of times
 around the periodic orbit of $p_0$,
and later comes back to the orbit of $p$.

Choosing the number of turns around the orbit of $p_0$ will 
allow to construct vectors of $W^{su}_+(p)$ arbitrarily close to any vector of 
the closed orbit $(g^tp)_{t\in\R}$. \\

First, the vector $v$ belongs to $W^{su}_+(p)$, and $p$ is periodic, 
so that there exists $\tau\ge 0$, satisfying
 $g^{-\tau}v\in W^{su}_\varepsilon(g^{-\tau}p)=W^{su}_\varepsilon(p)$. 
Choose the smallest such $\tau_1$, and let $v_1=g^{-\tau_1}v$. 
There exists $s_1\ge \tau_1$ 
s.t. $g^{s_1}v_1=g^{s_1-\tau_1}v\in W^{ss}_{\delta/2}(p_0)$;   
$s_1$ is  the ``time'' needed to come from an $\epsilon$-unstable neighbourhood 
of $p$ to a $\delta/2$-neighbourhood of the orbit of $p_0$, 
following the orbit of $v$. 
As $p_0$ is periodic, note that for all nonnegative integer $i\ge 0$, $g^{s_1+il(p_0)}v_1\in W^{ss}_{\delta/2}(p_0)$. 

There exists a vector $w\in W^{su}_{\delta/2}(p_0)\cap W^{ws}(p)$. 
Let $s_2>0$ s.t. $g^{s_2}w\in W^{ss}_{\varepsilon}(p)$. 

For all $k\in\N$, ass $d(w, g^{s_1+k.n.l(p_0)}v_1)\le \delta$, 
we use the local product structure of the geodesic flow, and
obtain a vector of $W^{su}_\varepsilon(g^{s_1+knl(p_0)\pm\varepsilon}v_1)\cap W^{ss}_\varepsilon(w)$.
The resulting geodesic orbit on $T^1 S$ is negatively and 
positively asymptotic to the orbit of  $p$, 
going from $p$ to $p_0$, $\varepsilon$-shadowing the orbit of $p_0$ during the
time $k.nl(p_0)$ and coming back to the orbit of $p$. 

The key point (and the only difference with \cite{C}) 
is that the ``gluing'' was done between some vector $g^tv$, $t\ge 0$
of the {\em positive} geodesic orbit of $v$,  and a vector ``coming back'' from $p_0$ to $p$. 
It ensures that the resulting geodesic orbit intersects $W^{su}_\varepsilon(g^tv)$. As 
$W^{su}_\varepsilon(v)$ contains $g^{-t}W^{su}_\varepsilon(g^tv)$,
 this orbit intersect therefore $W^{su}_\varepsilon(v)$  which is included in the positive 
unstable horocycle $W^{su}_+(p)$. 

Note that the time $s_1$ needed to go from  $p$ to $p_0$, 
and the time $s_2$ to come back, depend only on $\varepsilon$, and not on $k$, 
so that we can choose $k\in\N$ as large as we need. 

Let us repeat verbatim the final argument of \cite{C}. 
For all $\varepsilon>0$, we found $s_1>0$, $s_2>0$, 
s.t. for all $k\in\N$, there exists $u\in W^{ss}_{2\varepsilon}(p)$, 
and $s_k\in\R$, with $|s_k-s_1-s_2|<2\varepsilon$, and $g^{-s_k-knl(p_0)}u\in W^{su}_\varepsilon(p)$. 
Let $j\in\Z$ be the greatest integer such that $jl(p)<-s_k-knl(p_0)$. Then, 
$g^{jl(p)}u=g^{jl(p)+s_k+knl(p_0)}g^{-s_k-knl(p_0)}u \in W^{ss}(p)$ is 
$\varepsilon$-close to the vector $g^{s_k+kl(p_0)}p$ on the periodic orbit of $p$.
This vector also coincides with $g^{s_k+knl(p_0)+m'l(p)}p$ 
for all $m'\in\Z$. 
In particular, taking  $m'=km$, we find a vector on $W^{ss}(p)$ 
 very close to $g^{s_1+s_2+k(ml(p_0)+nl(p))}p$ for all positive integers $k\in\N$. 
As the length spectrum is non arithmetic,
 any point on the (periodic) geodesic orbit of $p$ is $\varepsilon$ close to such a point. 
Thus, $\overline{W^{su}_+(p)}$ contains $(g^tp)_{t\in\R}$, and therefore also $g^\R W^{su}_+(p)$ which is dense in $\Omega$.
This ends the proof.
\end{proof}

{\bf Proof of proposition \ref{right_iff_dense}}
The case of periodic vectors $p$ follows from proposition \ref{periodic} and the proof of 
lemma \ref{right-horospherical-gf}. We consider now only nonperiodic vectors. 

Assume first that $W^{su}_+(v)\cap\Omega$ is dense in $\Omega$, 
and prove that $v$ is a right horocyclic vector. 

Let $p$ be a vector on a periodic geodesic, $l(p)$ its length, 
and $d(p)$ the distance between $o$ and its orbit.
Fix $\alpha>0$ and $D>0$. 
Without loss of generality, we assume $D\ge l(p)+d(p)+2$. 
Consider the cone $\mathcal{C}=\mathcal{C}(g^{-D}\tilde{v},\alpha)$,
 where $\tilde{v}$ is a lift of $v$ to $T^1\D$.  
Remark that the distance between (the basepoint of) $h^s(g^{-D}\tilde{v})$ 
and the cone $\mathcal{C}$ goes to infinity when $s\to +\infty$.  

Fix $\varepsilon\in ]0,1[$. 
By density of $W^{su}_+(v)$ in $\Omega$, we can find an infinite sequence 
$v_k\in W^{su}_+(v)$, $v_k=h^{s_k}v$, $s_k\to \infty$, 
s.t. $v_k$ is so-close to $p$ that $(g^{-t}v_k)_{0\le t\le 2D}$ 
and $(g^{-t}p)_{0\le t\le 2D}$ stay $\varepsilon$-close each other. 
We deduce that $g^{-2D}v_k$ is at distance $\varepsilon$ from
$g^{-2D}p$, hence from the orbit of $p$, 
and therefore at distance less than $1+l(p)+d(p)$ 
from the projection $\pi(o)$ of $o$ on $S$. 
Lift $v$ to $\tilde{v}\in T^1\D$, and $v_k$ to $\tilde{v_k}\in W^{su}_+(\tilde{v})$
As $v_k=h^{s_k}v$ goes to infinity on $W^{su}_+(v)$, 
the distance between $g^{-2D}\tilde{v_k}$ and $\mathcal{C}$ goes to infinity. 
Therefore, we can  assume this distance be greater than $l(p)+d(p)+2$. 
There exists a point of $\Gamma.o$ at distance at most $d(p)+l(p)+1$ of $g^{-2D}v_k$. 
By construction, this point is inside $Hor^+(g^{-D}v)\setminus \mathcal{C}(v,\alpha)$.
This construction works for all $\alpha>0$ and $D>0$ large enough, 
so that $\tilde{v}$ is a right horospherical vector. \\

Let us establish now the other direction, adapting methods of \cite{C}. 
Let $v$ be a right horocyclic vector. 
We will prove that there exists a periodic vector $p\in \overline{W^{su}_+(v)}$, 
with $W^{su}_+(p)$ dense in $\Omega$. 

Let $t_n\to\infty$, $v_n\in W^{su}_+(v)\cap \Omega$, $v_n\to \infty$ on the leaf, 
s.t. $g^{-t_n}v_n$ converges to some $z\in\Omega$, with $g^{-t_n}\tilde{v}_n$
staying outside a given cone $\mathcal{C}(v,2)$. 
Let $p$ be a periodic vector s.t. $W^{su}_+(p)$ is dense in $\Omega$. 
Choose $\varepsilon_k\to 0$ and let $\delta_k$ be the constant associated to $\varepsilon_k$
 by the local product structure property around $z$. 

Using  this product structure,
 we construct an orbit negatively asymptotic to the negative orbit of $z$, 
and positively asymptotic to the orbit of $p$. More precisely, 
we can find $s_k\ge 0$, and $w_k\in W^{su}_{\delta_k/2}(z)\cap W^{ws}(p)$,
s.t. for all $t\ge s_k$, $g^{t}w_k$ is $\varepsilon_k$-close to the orbit of
$p$. Note that $w_k$ is $\delta_k/2$-close to $z$. 

Now, let $n_k$ be large enough so that $t_{n_k}\ge 2s_k$ and
 $d(g^{-t_{n_k}}v_{n_k},z)\le \delta_k/2$. 
In particular, the distance between $g^{-t_{n_k}}v_{n_k}$ 
and $w_k$ is at most $\delta_k$. 

As $v_{n_k}\in W^{su}_+(v)$ and $g^{-t_{n_k}}v_{n_k}$ is not in the cone $\mathcal{C}(v,2)$, 
 the local strong stable manifold 
$\displaystyle W^{su}_{2\varepsilon_k}(g^{-t_{n_k}}v_{n_k})$ is included in 
$\displaystyle W^{su}_+(g^{-t_{n_k}}v)$. 
This fact will be crucial  for the end of the proof; indeed,  we will now  glue 
the past orbit of $g^{-t_{n_k}}v_{n_k}$ with the future orbit of $w_k$,  
and the resulting orbit intersects 
the {\em positive} horocyclic orbit $W^{su}_+(v)$. Let us detail this gluing. 
Let $\tilde{v}$ be a lift of $v$, $\tilde{v}_{n_k}$ 
the lift of $v_{n_k}$ on $W^{su}_+(\tilde{v})$, 
$\tilde{z}_k$ (resp. $\tilde{w}_k$) be the lift of $z$ (resp $w_k$) 
$\delta_k/2$-close to $g^{-t_{n_k}}\tilde{v}_{n_k}$. 
Consider the geodesic joining $v^-$ to $\tilde{w}_k^+$.
 By the above, this geodesic crosses $W^{su}_{2\varepsilon_k}(g^{-t_{n_k}}v_{n_k})\subset W^{su}_+(g^{-t_{n_k}}v)$,
  and therefore also $W^{su}_+(\tilde{v})$. 

Let $\tilde{y}_k$ be
 the unique vector of $W^{su}_+(\tilde{v})$ on this geodesic. 
By construction $(g^{-t}\tilde{y}_k)_{t\ge 0}$ is asymptotic to $v^-$, 
and $g^{-t}\tilde{y}_k$
 belongs to  a $2\varepsilon_k$-neighbourhood of
 $\tilde{w}_k$ for $t\simeq t_{n_k}$, and then it becomes positively asymptotic to
 $(g^t\tilde{w}_k)_{t\ge -t_{n_k} }$. 
In particular, on $T^1\D$, as $s_k$ is the ``time'' 
needed on the orbit of $w_k$ to join the $\varepsilon_k$-neighbourhood 
of the orbit of $p$, for $t\ge s_k-t_{n_k}$,
 the orbit of $y_k$ becomes $2\varepsilon_k$-close to
 the orbit of $p$. 
We chose $t_{n_k}\ge 2s_k$ so that for $t=0$, $y_k$ is
 $2\varepsilon_k$-close to the orbit of $p$. 

As this orbit is a compact set, up to a subsequence, we can
 assume that $y_k$ converges. 
It implies  that there exists $0\le \sigma\le l(p)$ st $g^\sigma p\in \overline{W^{su}_+(v)}$. 
Of course $g^\sigma p$ is periodic and $W^{su}_+(g^\sigma p)$ is dense in $\Omega$. 
 Fact \ref{open} implies now that  $W^{su}_+(v)\cap \Omega$ is dense in $\Omega$. \\
\eop

{\bf Proof of lemma \ref{right-horospherical-gf}}
Assume first that $v^-$ is the first endpoint of an interval of $S^1\setminus \Lambda_\Gamma$. 
As the property of being right horospherical or not depends only on $v^-$, 
we can assume that $v$ is a periodic vector on the closed geodesic closing the funnel. 

By definition of a funnel, it becomes clear that if $o$ was chosen 
in a lift of the compact part of $S$, the intersection of the open right 
horoball $Hor^+(v)$ with $\Gamma.o$ is empty. 
Thus, $\tilde{v}$ is not a right horospherical vector. \\

Suppose now that $v$ is not a right horospherical vector. 
There exists a cone $\mathcal{C}(v,\alpha)$ and a right horoball $Hor^+(g^{-T}v)$, 
s.t. $\Gamma.o$ does not intersect $Hor^+(g^{-T}v)\setminus \mathcal{C}(v,\alpha)$. 
Let us shrink the horoball from a distance $d$ 
equal to the diameter of the compact part $C(S)$ of $S$. 
Thus, the set $Hor^+(g^{-T-d}v)\setminus\mathcal{C}(v,\alpha)$ 
does not intersect the $\Gamma$-orbit $\Gamma.\widetilde{C(s)}$ 
of the lift of the compact part. 
In other words, viewed on $S$, the projection of
 $Hor^+(g^{-T-d}v)\setminus\mathcal{C}(v,\alpha)$, which is a connected set,  
is necessarily included in a cusp or a funnel. 
It implies immediately that $v^-$ is a parabolic point or 
is the first  endpoint of an interval of $S^1\setminus \Lambda_\Gamma$.
 By assumption, $v^-$ cannot be parabolic, so that the result is proven. 
\eop

\subsection*{Geometrically infinite surfaces}

On these surfaces, the situation is -not surprisingly - more complicated, and 
we only discuss here partial results on the behaviour of positive (resp. negative) half-horocycles. 

Proposition \ref{periodic} gives a complete answer for periodic vectors. 
Recall the 
\begin{thm}[Hedlund, \cite{Hedlund}, thm 4.2 ] 
Let $S=\D/\Gamma$ be a hyperbolic surface of the first kind, i.e. such that $\Lambda_\Gamma=S^1$. 
Let $v\in T^1 S$ be s.t. $(g^{-t}v)_{t\ge 0}$ returns infinitely often in a compact set. 
Then the positive half-horocycle $(h^s v)_{s\ge 0}$ is dense in $T^1 S$.
\end{thm}

In \cite{SScha}, in the case of an abelian cover of a compact surface 
(also a surface of the first kind), 
we proved the equidistribution, and therefore the density of 
all positive half-horocyclic orbits $(h^sv)_{s\ge 0}$ of vectors $v$ 
whose asymptotic cycle is not maximal.

\begin{question}\rm It would be interesting to understand completely 
the behaviour of  half-horocycles. For example, 
\begin{enumerate}
\item On a surface of the first kind ($\Lambda_\Gamma=S^1$), are all horospherical vectors also right horocyclic vectors 
(generalization of Hedlund's theorem) ? Or can we find a counterexample ? 
\item On a surface of the second kind, 
can we construct counterexamples to Hedlund's theorem? 
Or sufficient conditions to be right horocyclic ? 
\end{enumerate}
\end{question}

\section{Proof of Theorem \ref{equidistribution}} \label{Proof-of-equidistribution}

In this section, $S$ is a nonelementary geometrically finite surface. 

\subsection*{Measures}

Let $\delta_{\Gamma}$ be the critical exponent of $\Gamma$, defined by
 $\delta_{\Gamma}:=\limsup_{T\to\infty}\#\{\gamma\in \Gamma, d(o,\gamma.o)\le T\}$. 
The well known Patterson construction provides a {\em conformal density}
 of exponent $\delta_{\Gamma}$ on $S^1$, 
that is a collection $(\nu_x)_{x\in\D}$ of  measures, 
supported on $\Lambda_{\Gamma}\subset S^1$, s.t. $\nu_o(S^1)=1$, $\gamma_*\nu_x=\nu_{\gamma.x}$ for all $\gamma\in\Gamma$, 
and $\frac{d\nu_x}{\d\nu_y}(\xi)=\exp(-\delta_{\Gamma}\beta_{\xi}(x,y))$. 

The Patterson-Sullivan measure $m^{ps}$ on  $T^1 S$, or Bowen-Margulis measure,  
is defined locally as the product 
$$
dm^{ps}(v)=
\exp\left(\delta_{\Gamma}\beta_{v^-}(o,\pi(v))+\delta_\Gamma\beta_{v^+}(o,\pi(v))\right)
d\nu_o(v^-)d\nu_o(v^+)dt
$$
in the coordinates $\Omega\simeq \Gamma\backslash (\Lambda_{\Gamma}^2 \times\R)$.

Under our assumptions on $S$, it is well known \cite{Sullivan} 
that the Bowen-Margulis measure is finite and ergodic\footnote{In fact, this result is false in general in 
variable negative curvature, and the assumption ($*$) 
added in section 5 ensures finiteness and ergodicity of this measure },
that there exists a unique conformal density of exponent $\delta_\Gamma$,
 that all measures $\nu_x$ are nonatomic and give full measure to the radial limit set. 
Moreover, the Bowen-Margulis-Patterson-Sullivan measure is the measure 
of maximal entropy of the geodesic flow,
 and it is fully supported on the nonwandering set $\Omega$ of the geodesic flow.

Denote by $\mu^{ps}_{H^+}$ the conditional measure of $m^{ps}$ 
on the strong unstable horocycle $H^+(u)=(h^su)_{s\in\R}$. 
It satisfies
 $d\mu^{ps}_{H^+}(v)=\exp\left(\delta_{\Gamma}\beta_{v^+}(o,\pi(v))\right)\,d\nu_o(v^+)$. 
To the measure $m^{ps}$ is also associated a {\em transverse measure} 
invariant by the horocyclic foliation, 
that is a collection $(\mu_T)$ of measures on all transversals $T$ to the
strong unstable foliation, 
invariant by all holonomies of the foliation. 

The classification of ergodic invariant measures for the horocyclic flow 
is well known (\cite{Burger},  \cite{Roblin}). 
Except the  probability measures supported on  periodic horocycles, 
and the infinite measures supported on  wandering horocycles, 
{\em there is a unique ergodic invariant measure fully supported in the
nonwandering set} 
$
{\mathcal
  E}\simeq \Gamma\backslash (\Lambda_\Gamma\times S^1\times \R)
$ of $(h^s)_{s\in\R}$. 
It is an infinite measure, defined locally by
$$
dm(v)=ds(v)\exp\left(\delta_{\Gamma}\beta_{v^-}(o,\pi(v))\right)d\nu_o(v^-)dt,
$$ 
where $ds(v)$ denotes the natural 
Lebesgue measure on $(h^s v)_{s\in\R}$ associated with the parametrization by $(h^s)$.

\subsection*{Sketch of the proof}

The strategy of the proof is exactly the same as in \cite{Scha1} and \cite{Scha2}. 
We consider 'one-sided versions' of all results of these articles. 
Due to the lengths of the proofs of technical results in \cite{Scha1}, 
we just recall the main arguments, and point out the few differences. \\

The main lines of the proof are as follows. 
We do not prove directly equidistribution of horocyclic orbits to the 
unique ``interesting'' ergodic invariant measure, because this measure is infinite. 
We consider auxiliary averages on horocycles. Using classical arguments 
(tightness in theorem \ref{nondivergence} and classification of invariant measures 
due to Burger \cite{Burger} and Roblin \cite{Roblin}), we prove equidistribution 
of these auxiliary averages towards the {\em finite} Bowen-Margulis measure 
(theorem \ref{equidistribution_to_Patterson_Sullivan}). 
We deduce then theorem \ref{equidistribution} from the preceding. \\

Let $\psi:T^1 S\to\R$ be a continuous compactly supported map. Denote by
$(h^su)_a^b$ the segment of orbit $(h^su)_{a\le s\le b}$. 
 Consider the following averages\,:
$$
M_{r,u}^+(\psi)=
\frac{1}{\mu^{ps}_{H^+}((h^su)_{0}^{R})}\int_{(h^su)_{0}^{R}}\psi(v)\,d\mu_{H^+}^{ps}(v)\,.$$
These averages are supported on $\Omega$. 
We prove 

\begin{thm}\label{equidistribution_to_Patterson_Sullivan}
 Let $S$ be a nonelementary geometrically finite hyperbolic surface, 
and $u\in{\mathcal E}\subset T^1 S$. 
If the positive orbit $(h^su)_{s\ge 0}$ is dense in $\Omega$, then it is equidistributed\,:
for all $\psi:T^1 S\to\R$ continuous with compact support, we have
$$
M_{r,u}^+(\psi)\to \int_{T^1 S}\psi\,dm^{ps}\,,\quad\mbox{when} \quad r\to \infty.
$$
\end{thm}

As in \cite{Scha2}, we deduce easily theorem \ref{equidistribution} of theorem 
\ref{equidistribution_to_Patterson_Sullivan}.
In the proof of  theorem \ref{equidistribution_to_Patterson_Sullivan}, 
the difficult parts are the classification of $(h^s)$-invariant measures (see
\cite{Burger} and \cite{Roblin}) and  the following  tightness argument. 
 
\begin{thm}\label{nondivergence}
  Let $S$ be a nonelementary geometrically finite hyperbolic surface, 
and $u\in{\mathcal E}\subset T^1 S$. 
For all $\varepsilon>0$, there exist a compact set $K_{\varepsilon,u}\subset
\Omega$ 
and $r_0>0$ such that for $r\ge r_0$, 
$M_{r,u}^+(K_{\varepsilon,u})\ge 1-\varepsilon$. 
\end{thm}

\subsection*{Proof of theorem \ref{nondivergence}}
This is the only non immediate part. 
For simplicity, assume that  $S$ has exactly one cusp. 
(If it has no cusp, $\Omega$ is compact, so that  theorem \ref{nondivergence} is obvious.)
Denote by $(\xi_i)_{i\in\N}=\Gamma.\xi_1$ the $\Gamma$-orbit of parabolic limit points of
$\Lambda_{\Gamma}$. 
As $S$ is geometrically finite, there is a $\Gamma$-invariant 
family of disjoint horoballs $H_i$ of $\D$, based at $\xi_i$,  
such that $\sqcup_{i\in\N}H_i=\Gamma.H_1$, $\Gamma$ acts cocompactly on $(\Lambda_{\Gamma}^2\times\R)\setminus \cup_i T^1 H_i$. 
Assume that $H_1$ is the closest horoball to the origin $o$, 
that the distance from $o$ to $\partial H_1$ is bounded by the diameter 
of the compact part $C_0$ of $S$, and that the geodesic ray $[o\xi_1)$ 
does not intersect other horoballs $H_i$, $i\neq 1$. 
Let $\Pi$ be the subgroup of $\Gamma$ stabilizing $H_1$. 
Its critical exponent $\delta_{\Pi}$ is equal to $1/2$ on hyperbolic surfaces. 
Moreover, $1>\delta_{\Gamma}>\delta_{\Pi}$ for a nonlattice not
convexe-compact geometrically finite group. (\footnote{Indeed, \cite{Peigne}
 as the surface has infinite volume, it contains a funnel. 
Let $\xi\notin\Lambda_\Gamma$, and $p\notin\Gamma$ a parabolic isometry fixing $\xi$. 
Using the divergence of $\Gamma$ and \cite[Prop.2]{DOP}, we obtain $1\ge \delta_{<p>*\Gamma}>\delta_\Gamma$. 
As all parabolic subgroups of $\Gamma$ are divergent, the same proposition \cite[Prop.2]{DOP} 
gives $\delta_\Gamma>\delta_\Pi$ for all $\Pi$ parabolic subgroups of $\Gamma$. 
  })

Let $o\in\D$ be fixed outside all horoballs $H_i$,  $\xi\in S^1$ and $t\ge 0$.
Let $\xi(t)$ be the point of the geodesic ray $[o\xi)$ at distance $t$ of $o$, and 
define the set $V(o,\xi,t)$ as the set of points $\eta\in S^1$ 
whose projection on $[o\xi)$ is at distance at least $t$ of $o$. 
By abuse of notation, we call such sets {\em shadows}, 
because they are comparable to Sullivan's shadows (it is a classical fact, see
for example \cite{Scha1}). 
We denote by $V(o,\xi,t)^+$ (resp. $V(o,\xi,t)^-$)  the {\em positive } (resp. negative) 
{\em half-shadow}, that is the subset of points of $V(o,\xi, t)$ that
are greater 
(resp. less) than $\xi$ in the counterclockwise order. 

If $H_i$ is a horoball based at $\xi_i$, denote by $s_i$ 
the distance between $o$ and $\partial H_i$, 
or in other words the instant when the geodesic ray $[o\xi_i)$ enters in $H_i$.
Notation $a(t)=B^{\pm 1}$ means that $\frac{1}{B}\le a(t)\le B$  for all $t\ge
0$. 

Following  exactly \cite[prop. 3.4]{Scha2}, we get : 

\begin{prop}\label{half_shadow} Let $S$ be a geometrically finite hyperbolic surface. 
There is a constant
 $B>0$ such that for all $\xi_i\in\Lambda_p$ and all $t\ge s_i$,
 where $s_i=\beta_{\xi_i}(o,\partial H_i)$, we have
$$
\nu_o(V(o,\xi_i, t)^+)=B^{\pm 1}e^{-\delta_{\Gamma}t}e^{(1-\delta_{\Gamma})(t-s_i)}\,.
$$
\end{prop}

We will need the following immediate refinement of the above statement. 
Let $s\ge 0$ be large enough so that $Be^{(1-2\delta_{\Gamma})s}\le \frac{1}{2B}$. 
We have
\begin{eqnarray}\label{couronne}\nu_o(V(o,\xi_i, t)^+\setminus V(o,\xi_i,
  t+s)^+)
=(2B)^{\pm 1}e^{-\delta_{\Gamma}t}e^{(1-\delta_{\Gamma})(t-s_i)}\,.
\end{eqnarray}
In other words, 'most' of the measure is in the boundary of the shadows. Of
course, the same result holds for $V(o,\xi_i, t)^-\setminus V(o,\xi_i,
t+s)^-$. It also holds, with a different constant, for
$V(o,\xi_i,t)^+\setminus V(o,\xi_i,t+s)$. 

Consider now a horocycle $(h^su)_{s\in\R}$, with $u^-\in\Lambda_{rad}$, 
$u\notin H_i$ for all $i\in\N$.
Denote by $v_i=h^{\sigma_i}u,\,i\in\N$ the unique vector of this horocycle 
such that $v_i^+=\xi_i$, 
and $h_i$ the height of $v_i$ in the horoball $H_i$ (i.e. the unique real
number 
s.t. $g^{-h_i}v_i\in T^1\partial H_i$).
 As noticed in \cite[lemma 4.3]{Scha2}, there is a (small) constant $\alpha>0$, s.t. 
$$
(h^sv_i)_{|s|\le e^{(h_i-\alpha)/2}}\subset (h^su)_{0}^{+\infty}\cap T^1
 H_i\subset (h^sv_i)_{|s|\le e^{h_i/2}}\,.
$$

For all $i\in\N$, define the horoball $H_i^N\subset H_i$ as the horoball s.t. 
the distance between $\partial H_i$ and $\partial H_i^N$ is equal to $N$. 
Using the above proposition, we prove that
\begin{lem}\label{uniformbound} Let $u\in {\mathcal E}$, with $u^-\in\Lambda_{rad}$.
 There is a
  constant $C>0$ such that   
$$
\frac{\mu_{H^+}^{ps}((h^s v_i)_{|s|\le e^{(h_i-N)/2}})}{\mu_{H^+}^{ps}((h^s
  v_i)_{e^{(h_i-N)/2}\le s\le e^{h_i/2}})} \le C e^{-(2\delta_\Gamma-1)N}\to 0 \quad
\mbox{when}\quad N\to\infty$$
 uniformly in $i\in\N$. 
\end{lem}
This lemma says that the 'time' (measured with the measure $\mu^{ps}_{H^+}$) 
spent by a horocycle in a horoball $H_i^N$ is
small compared to the 'time' needed to go from $\partial H_i^N$ to $\partial H_i^0$. 

Recall that as $S$ is a geometrically finite hyperbolic surface,
$\delta_\Gamma>\delta_\Pi=1/2$. 
\begin{proof} We only sketch the proof, and refer to \cite{Scha1}. Let $C_0$
  be the compact part of $S$, and $\tilde{C_0}$ a connected lift to $\D$
  containing $o$. 

As $u^-\in\Lambda_{rad}$, for all $i\in\N$, there exists $T_i\ge h_i/2$
s.t. $g^{-T_i}v_i\in\Gamma\tilde{C_0}$. By definition of $\mu_{H^+}^{ps}$ we get
$$
\frac{\mu_{H^+}^{ps}((h^s v_i)_{|s|\le e^{(h_i-N)/2}})}{\mu_{H^+}^{ps}((h^s
  v_i)_{ e^{(h_i-N)}\le s\le e^{h_i/2}})}= 
\frac{\mu_{H^+}^{ps}((h^s g^{-T_i}v_i)_{|s|\le e^{(h_i-N)/2-T_i}})}{\mu_{H^+}^{ps}((h^s
 g^{-T_i} v_i)_{e^{(h_i-N)/2-T_i}\le s\le e^{h_i/2-T_i}})}\,.
$$
As the distance from $\pi(g^{-T_i}v_i)$ to $\gamma.o$ is less than the
diameter of the compact part $C_0$, up to uniform
constants, the above quantity is uniformly close to 
$$
\frac{\nu_{\gamma.o}(V(\gamma.o,\xi_i,T_i-h_i/2+N/2))}
{\nu_{\gamma.o}\left(V(\gamma.o,\xi_i,T_i-h_i/2)^+\setminus
V(\gamma.o,\xi_i,T_i-h_i/2+N/2)^+\right)}\,.
$$  for some $\gamma\in \Gamma$. 
Proposition \ref{half_shadow} and estimate (\ref{couronne}) give the desired
control. 
\end{proof}

Following \cite{Scha1}, we can now prove theorem \ref{nondivergence}. Define
$\widetilde{K_{\varepsilon,u}}:=\Lambda_\Gamma^2\times \R\setminus
\sqcup_{i\in\N} T^1 H_i^N$, for $N=N(\varepsilon)$ large enough. 
Denote by $I_{u,r,N}$ the set of $i\in\N$ such that $(h^su)_{0}^R$ intersects
the unit tangent bundle $T^1 H_i^N$ to the shrinked horoball $H_i^N$ at height $N$ inside $H_i$. 
As $(h^su)_{s\in\R}$ is one-dimensional, and $u\in T^1C_0$,  
for all $j\in I_{u,r,N}$ except
maybe one boundary term denoted by $i_0=i_0(r)$, we have 
$(h^su)_0^r\cap T^1 H_i^N=(h^su)_{s\in\R}\cap
T^1 H_i^N$. 
 We deduce
\begin{eqnarray*}
M_{r,u}^+(\sqcup_{i\in\N} T^1 H_i^N)&\le& 
\frac{\mu_{H^+}^{ps}(\sqcup_{i\in I_{u,r,N}}(h^su)_0^r\cap T^1 H_i^N)}{\mu_{H^+}^{ps}(\sqcup_{i\in I_{u,r,N}}(h^su)_0^r\cap T^1 H_i^0)}\\
&\le &\frac{\sum_{j\in I_{u,r,N}\,, \,j\neq i_0(r)}\mu_{H^+}((h^sv_i)_{|s|\le
      e^{(h_i-N)/2}}) +\mu_{H^+}((h^s v_{i_0})_{|s|\le e^{(h_{i_0}-N)/2}})
    }
{\sum_{j\in I_{u,r,N},\,j\neq i_0(r)} 
\mu_{H^+}((h^sv_i)_{|s|\le    e^{h_i/2}})+\mu_{H^+}((h^s v_{i_0})_{-e^{h_i/2}}^{ -e^{(h_i-N)/2}})}
\end{eqnarray*}
Lemma \ref{uniformbound} allows to conclude that for $N$ large enough,
uniformly in $r\ge 0$, the measure $M_{r,u}^+(\Gamma.T^1 H_1^N)$ is less than
$\varepsilon$. 

Comparing with \cite{Scha1}, the only difference is  that here there
is only one boundary term $i_0(r)$.

\subsection*{Proof of theorem  \ref{equidistribution_to_Patterson_Sullivan}}

We follow \cite{Scha2}. Thanks to theorem \ref{nondivergence}, 
all limit points of $(M_{r,u}^+)_{r\ge 0}$ are probability measures. 
As in \cite[lemma 3.6]{Scha2}, we observe that such a limit gives measure 
zero to the set of periodic horocycles. 

Moreover, it is not difficult to see  \cite[lemma 3.5]{Scha2} 
that a limit point of the family $(M_{r,u}^+)_{r\ge 0}$
 when $r\to \infty$ can be written as the product 
of a transverse invariant measure to the strong unstable foliation 
by the measure $\mu_{H^+}^{ps}$. 
The only fact to check is that $\mu^{ps}_{H^+}((h^su)_{0\le s\le R})\to
+\infty$ as $R\to\infty$, and it can be done by the argument of  \cite[Lemma 4.2]{Scha2}, 
usin the fact that the positive half-horocycle $(h^su)_{s\ge 0}$ is dense. 

The uniqueness (\cite{Roblin}) of a transverse measure of full support in the
nonwandering set ${\mathcal E}$ giving measure $0$ to periodic horocycles
allows to conclude the proof.

\subsection*{Proof of theorem \ref{equidistribution}} 
As $(h^su)_0^{+\infty}$ is dense, and therefore recurrent, 
the proof is exactly the same as in
\cite{Scha2}. 
The idea is to restrict the attention to a small flow box $B$, and to compare
the transverse measures on a transversal $T$ of  $B$ induced on one side by the
averages $M_{r,u}^+(\psi)$ and on the other side by  ratios 
$\frac{\int_0^r\psi\circ h^s u\,ds}{\int_0^r{\bf 1}_B\circ h^s u\,ds}$.

\section{Surfaces with variable negative curvature}

Most results proved here extend to surfaces of variable negative curvature. 
More precisely, we assume that all sectional curvatures are pinched between 
two negative constants. 
Some definitions of the notions used here differ sightly, and we refer to the 
preliminary sections of \cite{Scha1} or \cite{Scha2} for details. 
The main difference is that there is no canonical parametrization of horocycles
 by a nice horocyclic flow, even if it is possible to define such a flow 
(see \cite{Marcus} ).

The motivated reader can check that the proof of theorem \ref{density} an all results of section 3 
extend verbatim to the situation of pinched negatively curved surfaces.\\


Concerning the equidistribution, we need to be more careful. 
We add an assumption, denoted by $(*)$ in \cite{Scha1} 
and \cite{Scha2}, which allows to control the geometry of the cusps, 
and ensures in particular that the Bowen-Margulis measure is finite. 
With this restriction, theorem \ref{equidistribution} is valid on pinched 
negatively curved geometrically finite surfaces.


\end{document}

%% file: produitlocal.pstex_t
\begin{picture}(0,0)%
\includegraphics{produitlocal.pstex}%
\end{picture}%
\setlength{\unitlength}{4144sp}%
\begingroup\makeatletter\ifx\SetFigFontNFSS\undefined%
\gdef\SetFigFontNFSS#1#2#3#4#5{%
  \reset@font\fontsize{#1}{#2pt}%
  \fontfamily{#3}\fontseries{#4}\fontshape{#5}%
  \selectfont}%
\fi\endgroup%
\begin{picture}(1232,1076)(7369,-4606)
\put(7954,-4133){\makebox(0,0)[lb]{\smash{{\SetFigFontNFSS{8}{9.6}{\familydefault}{\mddefault}{\updefault}{\color[rgb]{0,0,0}$u$}%
}}}}
\put(8138,-4160){\makebox(0,0)[lb]{\smash{{\SetFigFontNFSS{8}{9.6}{\familydefault}{\mddefault}{\updefault}{\color[rgb]{0,0,0}$v$}%
}}}}
\put(8073,-4008){\makebox(0,0)[lb]{\smash{{\SetFigFontNFSS{8}{9.6}{\familydefault}{\mddefault}{\updefault}{\color[rgb]{0,0,0}$[u,v]$}%
}}}}
\end{picture}%

%% file: horocyclenondense.pstex_t
\begin{picture}(0,0)%
\includegraphics{horocyclenondense.pstex}%
\end{picture}%
\setlength{\unitlength}{4144sp}%
\begingroup\makeatletter\ifx\SetFigFontNFSS\undefined%
\gdef\SetFigFontNFSS#1#2#3#4#5{%
  \reset@font\fontsize{#1}{#2pt}%
  \fontfamily{#3}\fontseries{#4}\fontshape{#5}%
  \selectfont}%
\fi\endgroup%
\begin{picture}(5356,1586)(12571,3310)
\put(16113,3959){\makebox(0,0)[lb]{\smash{{\SetFigFontNFSS{5}{6.0}{\familydefault}{\mddefault}{\updefault}{\color[rgb]{0,0,0}$v$}%
}}}}
\put(13608,4706){\makebox(0,0)[lb]{\smash{{\SetFigFontNFSS{5}{6.0}{\familydefault}{\mddefault}{\updefault}{\color[rgb]{0,0,0}$\tilde{v}$}%
}}}}
\put(13460,3366){\makebox(0,0)[lb]{\smash{{\SetFigFontNFSS{5}{6.0}{\familydefault}{\mddefault}{\updefault}{\color[rgb]{0,0,0}$v^-$}%
}}}}
\end{picture}%

%% file: right-horocyclic.pstex_t
\begin{picture}(0,0)%
\includegraphics{right-horocyclic.pstex}%
\end{picture}%
\setlength{\unitlength}{4144sp}%
\begingroup\makeatletter\ifx\SetFigFontNFSS\undefined%
\gdef\SetFigFontNFSS#1#2#3#4#5{%
  \reset@font\fontsize{#1}{#2pt}%
  \fontfamily{#3}\fontseries{#4}\fontshape{#5}%
  \selectfont}%
\fi\endgroup%
\begin{picture}(1320,1187)(8248,3933)
\put(8910,4973){\makebox(0,0)[lb]{\smash{{\SetFigFontNFSS{5}{6.0}{\familydefault}{\mddefault}{\updefault}{\color[rgb]{0,0,0}$\tilde{v}$}%
}}}}
\put(8980,4126){\makebox(0,0)[lb]{\smash{{\SetFigFontNFSS{5}{6.0}{\familydefault}{\mddefault}{\updefault}{\color[rgb]{0,0,0}$\gamma_n.o$}%
}}}}
\put(8837,3989){\makebox(0,0)[lb]{\smash{{\SetFigFontNFSS{5}{6.0}{\familydefault}{\mddefault}{\updefault}{\color[rgb]{0,0,0}$v^-$}%
}}}}
\put(9147,4341){\makebox(0,0)[lb]{\smash{{\SetFigFontNFSS{5}{6.0}{\familydefault}{\mddefault}{\updefault}{\color[rgb]{0,0,0}$\gamma_1.o$}%
}}}}
\end{picture}%

%% file: faiblestable.pstex_t
\begin{picture}(0,0)%
\includegraphics{faiblestable.pstex}%
\end{picture}%
\setlength{\unitlength}{4144sp}%
\begingroup\makeatletter\ifx\SetFigFontNFSS\undefined%
\gdef\SetFigFontNFSS#1#2#3#4#5{%
  \reset@font\fontsize{#1}{#2pt}%
  \fontfamily{#3}\fontseries{#4}\fontshape{#5}%
  \selectfont}%
\fi\endgroup%
\begin{picture}(959,671)(10345,-258)
\put(10843,198){\makebox(0,0)[lb]{\smash{{\SetFigFontNFSS{5}{6.0}{\familydefault}{\mddefault}{\updefault}{\color[rgb]{0,0,0}$\tilde{p}$}%
}}}}
\put(11102, 52){\makebox(0,0)[lb]{\smash{{\SetFigFontNFSS{5}{6.0}{\familydefault}{\mddefault}{\updefault}{\color[rgb]{0,0,0}$\tilde{v}_n$}%
}}}}
\put(11182,-83){\makebox(0,0)[lb]{\smash{{\SetFigFontNFSS{5}{6.0}{\familydefault}{\mddefault}{\updefault}{\color[rgb]{0,0,0}$x_n^+$}%
}}}}
\put(10809,-194){\makebox(0,0)[lb]{\smash{{\SetFigFontNFSS{5}{6.0}{\familydefault}{\mddefault}{\updefault}{\color[rgb]{0,0,0}$p^-$}%
}}}}
\end{picture}%